\documentclass[12pt, twoside, leqno]{article}

\usepackage[active]{srcltx}

\usepackage{amsfonts,amsmath,amsthm,amssymb}

\theoremstyle{plain}
\newtheorem{theorem}{Theorem}[section]

\newtheorem{corollary}[theorem]{Corollary}

\numberwithin{equation}{section}

\newcommand{\R}{\mathbb{R}}
\newcommand{\N}{\mathbb{N}}
\newcommand{\Z}{\mathbb{Z}}

\newcommand{\C}{\mathbb{C}}

\newcommand{\dis}{\displaystyle}

\begin{document}


\baselineskip=17pt


\title{Asymptotic distribution and symmetric means of algebraic numbers}

\author{Igor E. Pritsker\\
Department of Mathematics\\
Oklahoma State University\\
Stillwater, OK 74078,   U.S.A.\\
E-mail: igor@math.okstate.edu}

\date{}

\maketitle


\renewcommand{\thefootnote}{}



\renewcommand{\thefootnote}{\arabic{footnote}}
\setcounter{footnote}{0}


\begin{abstract}
Schur introduced the problem on the smallest limit point for the arithmetic means of totally positive conjugate algebraic integers. This area was developed further by Siegel, Smyth and others. We consider several generalizations of the problem that include questions on the smallest limit points of symmetric means. The key tool used in the study is the asymptotic distribution of algebraic numbers understood via the weak* limits of their counting measures. We establish interesting properties of the limiting measures, and find the smallest limit points of symmetric means for totally positive algebraic numbers of small height.
\end{abstract}

\section{Schur's problems on means of algebraic numbers}

Schur \cite{Sch} considered several problems for various classes of polynomials with integers coefficients and zeros restricted to certain sets. For convenience, we introduce the following notation. Let $E$ be a subset of the complex plane $\C.$ Consider the set of polynomials $\Z_n(E)$ with integer coefficients of the exact degree $n$ and all zeros in $E$. We denote the subset of $\Z_n(E)$ with simple zeros by $\Z_n^s(E)$. Given $M>0$, we write $P_n=a_nz^n + \ldots\in\Z_n^s(E,M)$ if $|a_n|\le M$ and $P_n\in\Z_n^s(E)$ (respectively $P_n\in\Z_n(E,M)$ if $|a_n|\le M$ and $P_n\in\Z_n(E)$). In particular, Schur \cite{Sch}, \S 4-8, studied the limit behavior of the arithmetic means of zeros for polynomials from $\Z_n^s(E,M)$ as $n\to\infty,$ where $M>0$ is an arbitrary fixed number. Two of his main results in this direction are stated below. Let $D:=\{z\in\C:|z|\le 1\}$ be the closed unit disk, and let $\R_+:=[0,\infty),$ where $\R$ is the real line. For a polynomial  $P_n(z)=a_n\prod_{k=1}^n (z-\alpha_{k,n})$, define the arithmetic mean of its zeros by $A_n:=\sum_{k=1}^n \alpha_{k,n}/n.$

\noindent{\bf Theorem A} (Schur \cite{Sch}, Satz XI) {\em If $P_n\in\Z_n^s(\R_+,M)$ is a sequence of polynomials such that $n\to\infty$, then}
\begin{align} \label{1.1}
\liminf_{n\to\infty} A_n \ge \sqrt{e} > 1.6487.
\end{align}

\medskip
\noindent{\bf Theorem B} (Schur \cite{Sch}, Satz XIII) {\em If $P_n\in\Z_n^s(D,M)$ is a sequence of polynomials such that $n\to\infty$, then}
\begin{align} \label{1.2}
\limsup_{n\to\infty} |A_n| \le 1-\sqrt{e}/2 < 0.1757.
\end{align}

Schur remarked that the $\limsup$ in \eqref{1.2} is equal to $0$ for {\em monic} polynomials from $\Z_n(D)$ by Kronecker's theorem \cite{Kr1857}. We proved \cite{PrCR} that $\dis\lim_{n\to\infty} A_n = 0$ for any sequence of polynomials from Schur's class $\Z_n^s(D,M),\ n\in\N.$ This result was obtained as a consequence of the asymptotic equidistribution of zeros near the unit circumference. Namely, if $\{\alpha_{k,n}\}_{k=1}^n$ are the zeros of $P_n$, we define the zero counting measure
\[
\tau_n := \frac{1}{n} \sum_{k=1}^n \delta_{\alpha_{k,n}},
\]
where $\delta_{\alpha_{k,n}}$ is the unit point mass at $\alpha_{k,n}$. Consider the normalized arclength measure $\mu_D$ on the unit circumference, with $d\mu_D(e^{it}):=\frac{1}{2\pi}dt.$ If $\tau_n$ converge to $\mu_D$ in weak* topology as $n\to\infty$ (written $\tau_n \stackrel{*}{\rightarrow} \mu_D$) then
\[
\lim_{n\to\infty} A_n = \lim_{n\to\infty} \int z\,d\tau_n(z) = \int z\,d\mu_D(z) = 0.
\]
Thus the sharp version of Schur's Theorem B immediately follows from the result $\tau_n \stackrel{*}{\rightarrow} \mu_D$ as $n\to\infty,$ which was originally proved in  \cite{PrCR}, and generalized in several directions in \cite{PrCrelle}, \cite{PrArk} and \cite{PrUMB}. Furthermore, we found essentially sharp rates of convergence for $A_n$ to 0 in the setting of Theorem B. Similar approach via the asymptotic distribution of algebraic numbers and limiting measures for $\tau_n$ can be used to develop Theorem A, but this problem is much more complicated because algebraic numbers are contained in the unbounded set $\R_+.$ We give a brief description of the history, referring to \cite{PrCrelle} for a more complete account. If $P_n(z)=a_{n,n} \prod_{k=1}^n (z-\alpha_{k,n})$ is irreducible over integers, then $\{\alpha_{k,n}\}_{k=1}^n$ is called a complete set of conjugate algebraic numbers of degree $n$. When $a_n=1$, we refer to $\{\alpha_{k,n}\}_{k=1}^n$ as algebraic integers. If $\alpha=\alpha_{1,n}$ is one of the conjugates, then the sum of $\{\alpha_{k,n}\}_{k=1}^n$ is also called the trace $\textup{tr}(\alpha)$ of $\alpha$ over rationals. Siegel \cite{Si} improved Theorem A for totally positive algebraic integers to
\[
\liminf_{n\to\infty} A_n = \liminf_{n\to\infty} \textup{tr}(\alpha)/n > 1.7336105,
\]
by using a refinement of the arithmetic-geometric means inequality that involves the discriminant of $\alpha_{k,n}$. Smyth \cite{Sm2} introduced a numerical method of ``auxiliary polynomials," which was used by many authors to obtain improvements of
the above lower bound. The original papers \cite{Sm1,Sm2} contain the bound $1.7719$. More recent results include bounds $1.784109$ by Aguirre and Peral \cite{AP}, and $1.78702$ by Flammang \cite{Fla}. McKee \cite{McKee} designed a modification of the method that achieves the bound $1.78839$, which is apparently the best currently known lower bound.

\smallskip
\noindent\textbf{Problem A} (The Schur-Siegel-Smyth trace problem \cite{Bor02}) \textit{Find the smallest limit point $\ell$ for the set of values of mean traces $A_n$ for all totally positive and real algebraic integers.}
\smallskip

It was observed by Schur \cite{Sch} (see also Siegel \cite{Si}), that $\ell \le 2$. This immediately follows by considering the Chebyshev polynomials $t_n(x):=2\cos(n\arccos((x-2)/2))$ for the segment $[0,4]$, whose zeros are symmetric about the midpoint $2$. They have integer coefficients, and $t_p(x)/(x-2)$ is irreducible for any prime $p$, giving the mentioned upper bound 2, cf. \cite{Sch}. In fact, we have for the counting measures of zeros $\tau_n \stackrel{*}{\rightarrow} dx/(\pi\sqrt{x(4-x)})$ as $n\to\infty$, so that
\[
\lim_{n\to\infty} \frac{\textup{tr}(\alpha)}{n} = \lim_{n\to\infty} \int x\,d\tau_n(x) = \int_0^4 \frac{x\,dx}{\pi\sqrt{x(4-x)}} = 2.
\]
More generally, if a sequence of totally positive algebraic numbers satisfies $\tau_n \stackrel{*}{\rightarrow} \mu$ as $n\to\infty$, then
\[
\liminf_{n\to\infty} \frac{\textup{tr}(\alpha)}{n} = \liminf_{n\to\infty} \int x\,d\tau_n(x) \ge \int x\,d\mu(x).
\]
This brings us to the problem of minimizing the centroids (first moments) of measures arising as weak* limits for the counting measures of totally positive algebraic numbers. We develop and generalize this approach so that it applies to a wide range of problems on symmetric and convex means of algebraic numbers. It may also be useful for other types of problems on algebraic numbers.

The Schur-Siegel-Smyth trace problem remains a difficult open question despite all efforts. As a partial result towards this problem, we gave the sharp lower bound $\liminf_{n\to\infty} A_n\ge 2$ for sets of algebraic numbers whose polynomials do not grow exponentially fast on compact sets of $\R_+$ of capacity (transfinite diameter) $1$, see Corollary 2.6 of \cite{PrCrelle} and Theorem \ref{thm2.6} here.

We use the \emph{generalized Mahler measure} to measure the size of integer polynomials over a certain set. The classical Mahler measure of a polynomial $P_n(z) = a_n\prod_{k=1}^n (z-\alpha_{k,n}),\ a_n\neq 0,$ is defined by
\[
M(P_n) := \exp\left(\frac{1}{2\pi} \int_0^{2\pi} \log |P_n(e^{it})|\,dt\right)=|a_n|\prod_{k=1}^n \max(1,|\alpha_{k,n}|),
\]
where the last equality is a consequence of Jensen's formula. The Mahler measure was generalized to compact sets of capacity 1 by Rumely \cite{Ru}. We employ a similar generalization, which was introduced in \cite{PrCrelle} to obtain an ``if and only if" theorem on the equidistribution of algebraic numbers near arbitrary compact sets in the plane.

Consider an arbitrary compact set $E\subset\C$ with logarithmic capacity cap$(E) = 1,$ see \cite[p. 55]{Ts}. In particular, cap$(D) = 1$ and capacity of a segment is equal to one quarter of its length \cite[p. 84]{Ts}. Let $\mu_E$ be the equilibrium measure of $E$ \cite[p. 55]{Ts}, which is a unique probability measure expressing the steady state distribution of charge on the conductor $E$. Note that $\mu_E$ is supported on the boundary of the unbounded connected component $\Omega_E$ of $\overline{\C}\setminus E$ by \cite[p. 79]{Ts}. Two important examples for us are $d\mu_D(e^{it})=\frac{1}{2\pi}dt$ and
\begin{align*}
d\mu_{[0,4]}(x)=\frac{dx}{\pi\sqrt{x(4-x)}}, \ x\in(0,4).
\end{align*}

Consider the Green function $g_E(z,\infty)$ for $\Omega_E$ with pole at $\infty$ (cf. \cite[p. 14]{Ts}), which is a positive harmonic function in $\Omega_E\setminus\{\infty\}$. Note that $g_D(z,\infty)=\log|z|,\ |z|>1,$ and $g_{[0,4]}(z,\infty)=\log|z-2+\sqrt{z^2-4z}|-\log{2},\ z\in\C\setminus[0,4].$ A natural generalization of the Mahler measure for $P_n(z) = a_n\prod_{k=1}^n (z-\alpha_{k,n}),\ a_n\neq 0,$ on an arbitrary compact set $E$ of capacity 1, is given by
\begin{align} \label{1.3}
M_E(P_n) := |a_n| \exp\left(\sum_{\alpha_{k,n}\in\Omega_E} g_E(\alpha_{k,n},\infty)\right).
\end{align}
If no $\alpha_{k,n}\in\Omega_E$ then we assume that the above (empty) sum is equal to zero. It is clear that $M_E(P_n) \ge |a_n| \ge 1$ for any $P_n$ with integer coefficients and $a_n\neq 0.$ If $M_E(P_n)$ does not grow fast with $n\to\infty$, which is specifically expressed by condition \eqref{2.9}, then the roots of $P_n$ are equidistributed in $E$, meaning that $\tau_n \stackrel{*}{\rightarrow} \mu_E$. This allows us to obtain the lower bound \eqref{2.10} of Theorem \ref{thm2.6} for the symmetric means of such algebraic numbers.

The elementary symmetric functions in the roots of $P_n(z) = a_n \prod_{k=1}^n (z-\alpha_{k,n}) = \sum_{k=0}^n a_k z^k$ are expressed through the coefficients by
\begin{equation}\label{1.4}
\sigma_m := \sum_{j_1<j_2<\ldots<j_m} \alpha_{j_1,n} \alpha_{j_2,n} \ldots  \alpha_{j_m,n} = (-1)^m\, \frac{a_{n-m}}{a_n}.
\end{equation}
Thus the Schur-Siegel-Smyth problem is equivalent to a statement on the growth of $|a_{n-1}|$ with $n$, and it is natural to consider higher symmetric functions $\sigma_m$ and the asymptotic behavior of the coefficients $a_{n-m}$ for a fixed $m\in\N$ when $n\to\infty.$

\section{Symmetric means of algebraic numbers}

For motivation and clarity of presentation, we first restrict ourselves to monic polynomials, following Siegel \cite{Si}. Thus we assume for a moment that $P_n(z)=z^n+a_{n-1,n}z^{n-1}+\ldots+a_{0,n} \in \Z_n^s(\R_+,1).$ Observe that each $\sigma_m$ has $\binom{n}{m}$ number of products in the defining sum \eqref{1.4}. Thus it is natural to consider the symmetric means
\[
S_m(x_1,\ldots,x_n):= \binom{n}{m}^{-1} \sum_{j_1<j_2<\ldots<j_m} x_{j_1} x_{j_2} \ldots  x_{j_m}, \quad m\in\N.
\]
The inequalities of Maclaurin (cf. Section 52 of \cite{HLP}) give fundamental relations between symmetric means of nonnegative numbers $\{x_i\}_{i=1}^n$:
\begin{align} \label{2.1}
\left(S_n\right)^{1/n} \le \ldots \le \left(S_2\right)^{1/2} \le S_1.
\end{align}
Note that equality holds in \eqref{2.1} if and only if all numbers $x_i$ are equal. Thus in the context of Schur's problems and their generalizations we always have strict inequalities in \eqref{2.1}. Furthermore, \eqref{2.1} gives for totally positive algebraic integers that
\[
S_m \ge (S_n)^{m/n} = \left(\prod_{k=1}^n \alpha_{k,n}\right)^{m/n} = \left|a_0\right|^{m/n} \ge 1,\quad 1\le m\le n.
\]
For any sequence of polynomials $P_n(z)=\dis\sum_{k=0}^n a_k z^k \in\Z_n^s(\R_+,1)$, we obtain that
\[
\liminf_{n\to\infty} \frac{|a_{n-m,n}|}{\binom{n}{m}} = \liminf_{n\to\infty} \frac{\sigma_m}{\binom{n}{m}} =  \liminf_{n\to\infty} S_m \ge 1,
\]
where we assume that $m\in\N$ is fixed. This suggests a generalization of the Schur-Siegel-Smyth trace problem:

\smallskip
\noindent\textbf{Problem B} \textit{Given a fixed $m\in\N$, find the smallest limit point $\ell_m$  for the symmetric means $S_m$ of roots of polynomials $P_n\in\Z_n^s(\R_+,1)$.}
\smallskip

Clearly, we have that $\ell_1=\ell$. We shall investigate relations between $\ell_m,\ m\ge 2,$ and $\ell$ in the present paper. An immediate consequence of \eqref{2.1} is that $\ell_m \le (\ell_k)^{m/k} \le \ell^m$ for $m\le k\le 1.$ On the other hand, Theorem 2.7 of \cite{PrUMB} suggests the conjecture $\ell_m=\ell^m,\ m\in\N.$ We give more evidence in support of this conjecture below.

We continue developing the approach to extremal problems on means of algebraic numbers via the integrals of limiting measures (for the counting measures of those numbers). Any sequence of the zero counting measures $\tau_n$ has a weak* convergence subsequence by Helly's Selection Theorem. While the classical version of Helly's theorem requires the measures be supported in a fixed compact set, we can apply the result on each closed disk $\{z:|z|\le n\},\ n\in\N,$ and construct nested subsequences to obtain a subsequence of measures weak* convergent on the union of expanding disks that fills the whole plane $\C.$ Details of this argument may be found, for example, in \cite[\S 1]{La}. Note that the limiting measure $\tau$ obtained in this way may have unbounded support and may have total mass $\tau(\C)<1.$ One of our main goals is to understand the properties of such weak* limits that shed new light on the difficult problems related to algebraic numbers. A natural class of limiting measures for these problems is described below. Define the logarithmic energy of a measure $\mu$ by
\[
I[\tau] := \iint \log\frac{1}{|z-t|}\,d\mu(z)d\mu(t),
\]
see \cite[p. 54]{Ts}. We state the first result under the general assumption that algebraic numbers have bounded Weil height (absolute Mahler measure).

\begin{theorem} \label{thm2.1}
 Suppose that $P_n(z) = a_n\prod_{k=1}^n (z-\alpha_{k,n})\in\Z_n^s(\C),\ n\in\N,$ is a sequence of polynomials with $\tau_n\stackrel{*}{\rightarrow}\tau$ along a subsequence ${\mathcal N}\subset\N.$  If
\begin{align} \label{2.2}
H:=\limsup_{{\mathcal N}\ni n\to\infty} \left(M(P_n)\right)^{1/n} < \infty
\end{align}
then $\tau$ is a unit measure with finite logarithmic energy. Furthermore, we have that
\begin{align} \label{2.3}
\int\log^+|z|\,d\tau(z) \le \log{H}
\end{align}
and the restriction $\tau_R:=\tau\vert_{\overline D_R},$ where $D_R:=\{z:|z|<R\},$ satisfies
\begin{align} \label{2.4}
- \log{2} - 2\tau(\overline D_R))\log{H} \le I[\tau_R] \le (1-\tau(D_R))\log{4} + 2 \log{H}
\end{align}
for all but countably many $R>1.$
\end{theorem}

Thus no loss of mass for the limiting measure occurs in this case, even though some algebraic numbers may tend to infinity as the degree $n$ increases. Furthermore, finiteness of the energy $I[\tau]$ carries information on the distribution of algebraic numbers. In particular, it implies that $\tau$ has no point masses, reflecting the fact that the algebraic numbers are well spaced. We remark that \eqref{2.2} is equivalent to
\begin{align} \label{2.5}
\limsup_{{\mathcal N}\ni n\to\infty} |a_n|^{1/n} < \infty \quad \mbox{and} \quad \limsup_{{\mathcal N}\ni n\to\infty} \left(\prod_{k=1}^n \max(1,|\alpha_{k,n}|)\right)^{1/n} < \infty.
\end{align}
Theorem \ref{thm2.1} has immediate applications to problems on symmetric means of algebraic numbers.

\begin{corollary} \label{cor2.2}
If \eqref{2.2} is replaced with
\begin{align} \label{2.6}
\limsup_{{\mathcal N}\ni n\to\infty} |a_n|^{1/n} < \infty \quad  \mbox{and} \quad  \limsup_{{\mathcal N}\ni n\to\infty} S_m(|\alpha_{1,n}|,\ldots,|\alpha_{n,n}|) < \infty
\end{align}
for a fixed $m\in\N,$ and if all other assumptions of Theorem \ref{thm2.1} hold, then $\tau$ is a unit measure with finite logarithmic energy.
\end{corollary}

On the other hand, we show that omitting assumptions may lead to essentially arbitrary limiting distribution of algebraic numbers.

\begin{theorem} \label{thm2.3}
Given any positive Borel measure $\mu,\ 0 \le \mu(\C) \le 1,$ that is symmetric about real line, there is a sequence of complete sets of conjugate algebraic integers such that their counting measures $\tau_n$ converge weak* to $\mu.$
\end{theorem}

The above irregular behavior of algebraic numbers may only occur when the height grows super exponentially and some of the conjugates escape to infinity as the degree increases, according to Theorem \ref{thm2.1}. As an example of conjugate algebraic integers with the counting measures converging to the identically zero measure in the weak* topology, we mention the roots of $P_p(z)=z^p-p!$ for prime $p$.

If all algebraic numbers are uniformly bounded for all $n\in\N$, then we can prove the conjectured relation between the limits of symmetric means.

\begin{theorem} \label{thm2.4}
If the numbers $\{z_{k,n}\}_{k=1}^n\subset\C$ are uniformly bounded, and their counting measures satisfy $\tau_n\stackrel{*}{\rightarrow}\tau$ for $n\in{\mathcal N}\subset\N$, then
\begin{align} \label{2.7}
\lim_{{\mathcal N}\ni n\to\infty} S_m(z_{1,n},\ldots,z_{n,n}) &= \lim_{{\mathcal N}\ni n\to\infty} \left(S_1(z_{1,n},\ldots,z_{n,n})\right)^m \\ &= \left(\int z\,d\tau(z)\right)^m, \quad m\in\N. \nonumber
\end{align}
\end{theorem}

In the case of unbounded numbers, we prove the following lower bound for their symmetric means. It is convenient to consider numbers located in closed sectors $V_m:=\{z\in\C:|\textup{Arg}\,z|\le \pi/(2m)\},$ where $m\in\N.$

\begin{theorem} \label{thm2.5}
For $m\in\N$, let the points $\{z_{k,n}\}_{k=1}^n\subset V_m$ be symmetric about the real line for all $n\in{\mathcal N}\subset\N$. If the counting measures  $\tau_n\stackrel{*}{\rightarrow}\tau,\ n\in{\mathcal N}$, where $\tau(\C)=1,$ then
\begin{align} \label{2.8}
\liminf_{{\mathcal N}\ni n\to\infty} S_m(z_{1,n},\ldots,z_{n,n}) \ge \left(\int z\,d\tau(z)\right)^m.
\end{align}
\end{theorem}

It is clear from weak* convergence that $\tau$ is symmetric in the real axis, and is supported in $V_m$ by inheritance. Therefore, we have  $\int z\,d\tau(z) \ge 0.$ Note that \eqref{2.8} holds for all $m\in\N$ when all $\{z_{k,n}\}_{k=1}^n\subset R_+.$

Let $\nu_m$ be a weak* limit of counting measures for algebraic integers that produce the smallest limit point $\ell_m$ of the symmetric means $S_m,\ m\in\N.$ We have that
\[
\ell_m \ge \left(\int x\,d\nu_m(x)\right)^m, \quad m\in\N,
\]
by \eqref{2.8}. It is plausible that equality holds above for all $m\in\N,$ but we were not able to prove this. Furthermore, if $\ell=\int x\,d\nu_m(x)$ then the above inequality and \eqref{2.1} give that $\ell_m=\ell^m.$ While we cannot show this holds in general, we can handle an important case of totally positive algebraic numbers with relatively small generalized Mahler measure (low height).

\begin{theorem} \label{thm2.6}
Let $P_n(x) = a_n\prod_{k=1}^n (x-\alpha_{k,n})\in\Z_n^s(\R_+),\ n\in{\mathcal N}\subset\N,$ be a sequence of polynomials, and let $E\subset\R_+$ be a compact set of logarithmic capacity 1. If
\begin{align} \label{2.9}
\lim_{{\mathcal N}\ni n\to\infty} \left(M_E(P_n)\right)^{1/n} = 1
\end{align}
then
\begin{align} \label{2.10}
\liminf_{{\mathcal N}\ni n\to\infty} S_m(\alpha_{1,n},\ldots,\alpha_{n,n}) \ge 2^m, \quad m\in\N.
\end{align}
Equality holds in \eqref{2.10} for the roots of $t_n(x)=2\cos\left(n\arccos((x-2)/2)\right).$
\end{theorem}

The key fact used in the proof of this result is that the roots of $P_n$ are equidistributed in $E$ under assumption \eqref{2.9}, that is $\tau_n \stackrel{*}{\rightarrow} \mu_E$ by Theorem 2.1 of \cite{PrCrelle}.

\section{Convex means of totally positive algebraic numbers}

Suppose that $\phi:\R_+\to\R$ is a convex strictly increasing function. We now consider the convex $\phi$-means of numbers $\{x_k\}_{k=1}^n\subset\R_+$ defined by
\begin{align} \label{3.1}
C_n^{\phi}(x_1,\ldots,x_n) := \frac{1}{n} \sum_{k=1}^n \phi(x_k).
\end{align}
Another interesting question generalizing the original Schur-Siegel-Smyth problem is to find the smallest limit point of such convex $\phi$-means for totally positively algebraic integers. This problem for means of squares of algebraic numbers was already considered in the original paper of Schur \cite{Sch}, and means for $\phi(x)=x^m$ were studied in the papers of Smyth \cite{Sm1} and \cite{Sm2}. Since $\phi$ is convex increasing, one immediately finds that
\begin{align} \label{3.2}
\liminf_{n\to\infty} C_n^{\phi}(\alpha_{1,n},\ldots,\alpha_{n,n}) \ge \liminf_{n\to\infty} \phi\left(\frac{1}{n} \sum_{k=1}^n \alpha_{k,n}\right) \ge \phi(\ell)
\end{align}
 for any sequence of complete sets of totally positive conjugate algebraic integers $\{\alpha_{k,n}\}_{k=1}^n,\ n\in\N$. The following result was proved in \cite{PrCrelle}, see Corollary 2.6 in that paper. It is based on the equidistribution of algebraic numbers in $E$ expressed by $\tau_n \stackrel{*}{\rightarrow} \mu_E$, and the lower bound for the centroid of equilibrium measure $\mu_E$ found in Theorem 1 of Baernstein, Laugesen and Pritsker \cite{BLP}.

\begin{theorem} \label{thm3.1}
Let $P_n(z) = a_n\prod_{k=1}^n (z-\alpha_{k,n})\in\Z_n^s(\R_+)$ be a sequence of polynomials. Suppose that $E\subset\R_+$ is a compact set of capacity $1$. We also assume that $\phi:\R_+\to\R_+$, and that $\phi(x^2)$ is convex on $\R$. If
\[
\lim_{n\to\infty} \left(M_E(P_n)\right)^{1/n} = 1
\]
then
\[
\liminf_{n\to\infty} \frac{1}{n} \sum_{k=1}^n \phi(\alpha_{k,n}) \ge \int_0^4 \frac{\phi(x)\,dx}{\pi\sqrt{x(4-x)}}.
\]
Setting $\phi(x)=x^m,\ m\in\N,$ we obtain
\[
\liminf_{n\to\infty} \frac{1}{n} \sum_{k=1}^n \alpha_{k,n}^m \ge \int_0^4 \frac{x^m\,dx}{\pi\sqrt{x(4-x)}} = 2^m \frac{1\cdot 3\cdot\ldots\cdot(2m-1)}{m!}.
\]
\end{theorem}

Equalities hold in the inequalities of Theorem \ref{thm3.1} for the roots of $t_n(x)=2\cos\left(n\arccos((x-2)/2)\right).$

 From a more general perspective, we show that the limiting measures for algebraic numbers with bounded convex means have finite energy as a consequence of Corollary \ref{cor2.2}.

 \begin{theorem} \label{thm3.2}
Let $P_n(z) = a_n\prod_{k=1}^n (z-\alpha_{k,n})\in\Z_n^s(\C),\ n\in\N,$ be a sequence of polynomials with $\tau_n\stackrel{*}{\rightarrow}\tau$ along a subsequence ${\mathcal N}\subset\N.$  If
\begin{align} \label{3.3}
\limsup_{{\mathcal N}\ni n\to\infty} |a_n|^{1/n} < \infty \quad  \mbox{and} \quad \limsup_{{\mathcal N}\ni n\to\infty} C_n^{\phi}(|\alpha_{1,n}|,\ldots,|\alpha_{n,n}|) < \infty,
\end{align}
then $\tau$ is a unit measure with finite logarithmic energy.
\end{theorem}

\section{Means of algebraic numbers located in sectors}

It is also possible to obtain lower bounds for the means of zeros for polynomials with integer coefficients and zeros in sectors $W_{\gamma}:=\{z\in\C:|\textup{Arg}\,z|\le \gamma\},$ where $\gamma<\pi/2.$ For the arithmetic means $S_1,$ one can easily give the following bound:
\begin{align*}
\frac{1}{n} \sum_{k=1}^n \alpha_{k,n} &= \frac{1}{n} \sum_{k=1}^n \Re(\alpha_{k,n}) \ge \frac{1}{n} \sum_{k=1}^n |\alpha_{k,n}|\cos\gamma \\ &\ge \left(\frac{|a_0|}{|a_n|}\right)^{1/n} \cos\gamma \ge \frac{\cos\gamma}{|a_n|^{1/n}}. \nonumber
\end{align*}
If \eqref{2.6} is satisfied, then we have a positive lower bound in the above inequality. Flammang \cite{FlS} applied Smyth's method of auxiliary functions to bound the smallest limit points of mean trace for algebraic integers located in sectors, and obtained many explicit results. In fact, \cite{FlS} gives exact values of smallest limit points for some sectors.

Note that Theorems \ref{thm2.1}, \ref{thm2.4} and Theorem \ref{thm2.5}, as well as Corollary \ref{cor2.2}, are clearly applicable here. In particular, combining Corollary \ref{cor2.2} with Theorem \ref{thm2.5}, we obtain the following result on the smallest limit points of means.

\begin{theorem} \label{thm4.1}
Let $P_n(z) = a_n\prod_{k=1}^n (z-\alpha_{k,n})\in\Z_n^s(W_{\gamma}),\ n\in{\mathcal N}\subset\N,$ be any sequence of polynomials, and let $\gamma<\pi/2.$  If \eqref{2.6} holds then
\begin{align} \label{4.1}
\liminf_{{\mathcal N}\ni n\to\infty} S_m(\alpha_{1,n},\ldots,\alpha_{n,n}) \ge \int z\,d\tau(z)>0, \quad \gamma\le \pi/(2m).
\end{align}
where each $\tau$ is a unit measure with finite logarithmic energy, which may be different for different $m\in\N.$
\end{theorem}

Considering the whole family of monic polynomials from $\Z_n^s(W_{\gamma})$, we can always select a subsequence whose symmetric means provide the smallest limit point, and for which $\tau_n\stackrel{*}{\rightarrow}\tau$. Then \eqref{2.6} is satisfied and \eqref{4.1} gives us a (theoretic) lower bound for that limit point.

\section{Proofs}

\begin{proof}[Proof of Theorem \ref{thm2.1}]

For notational convenience, we assume that $\tau_n$ converge to $\tau$ in the weak* topology along the whole sequence $\N.$  Our first step is proving that $\tau$ is a unit measure. If $\tau(\C)=\lambda<1$ then
\[
\limsup_{n\to\infty} \tau_n(\overline D_R) \le \tau(\overline D_R) \le \lambda
\]
for any disk $D_R:=\{z:|z|< R\}$, by the weak* convergence properties of Theorem 2.1 \cite[p. 16]{Bi}. Hence
\[
\liminf_{n\to\infty} \tau_n(\C\setminus \overline D_R) \ge 1-\lambda > 0 \quad \mbox{for all } R>0.
\]
Moreover, for any $\varepsilon>0$ and any $R>1$ there is $N\in\N$ such that
\[
M(P_n) \ge \prod_{k=1}^n \max(1,|\alpha_{k,n}|) \ge \prod_{|\alpha_{k,n}|> R} |\alpha_{k,n}| \ge R^{n(1-\lambda-\varepsilon)}, \quad n\ge N.
\]
The latter inequality is clearly incompatible with \eqref{2.2} when $R\to\infty$, so that $\tau(\C)=1$ follows. Note that $\tau$ may have unbounded support.


Our goal now is to show that the restriction of $\tau$ to every disk has finite energy. Given $R>1$, we define
\[
\hat\tau_n := \tau_n\vert_{\overline D_R} = \frac{1}{n} \sum_{|\alpha_{k,n}|\le R} \delta_{\alpha_{k,n}}.
\]
Note that $\hat\tau_n \stackrel{*}{\rightarrow} \tau_R:=\tau\vert_{\overline D_R}$ as $n\to\infty$ for any $R>0$ such that $\tau(\partial D_R)=0$, by Theorem $0.5^\prime$ of \cite[p. 10]{La}. Hence $\hat\tau_n \stackrel{*}{\rightarrow} \tau_R$ for all but countably many $R>1$, and we consider only such $R$ below. Furthermore, it is sufficient to consider $R$ such that $\tau(\overline D_R)>0$, for otherwise $I[\tau_R]=0$ and \eqref{2.4} holds trivially. Let $\Delta(P_n)=a_n^{2n-2} (V(P_n))^2$ be the discriminant of $P_n$, where
\[
V(P_n):=\prod_{1\le j<k\le n} (\alpha_{j,n}-\alpha_{k,n})
\]
is the Vandermonde determinant. Since $P_n$ has integer coefficients, $\Delta(P_n)$ is an integer, see \cite[p. 24]{Pra}. As $P_n$ has simple roots, we obtain that $\Delta(P_n)\neq 0$ and $|\Delta(P_n)|\ge 1.$ We now order $\alpha_{k,n}$ as follows
\[
|\alpha_{1,n}| \le |\alpha_{2,n}| \le \ldots \le |\alpha_{m_n,n}| \le R < |\alpha_{m_n+1,n}| \le \ldots \le |\alpha_{n,n}|.
\]
Let $\hat P_n(z) := a_n \prod_{k=1}^{m_n} (z-\alpha_{k,n})$ and $V(\hat P_n)=\prod_{1\le j<k\le m_n} (\alpha_{j,n}-\alpha_{k,n}).$
Hence
\begin{align} \label{5.1}
1 &\le |\Delta(P_n)| = |a_n|^{2n-2} |V(P_n)|^2  = |a_n|^{2n-2} |V(\hat P_n)|^2 \prod_{1\le j<k \atop m_n<k\le n} |\alpha_{j,n}-\alpha_{k,n}|^2 \nonumber \\ &\le |a_n|^{2n-2} |V(\hat P_n)|^2 \prod_{m_n<k\le n} (2|\alpha_{k,n}|)^{2(n-1)} \nonumber \\  &\le |V(\hat P_n)|^2 4^{(n-1)(n-m_n)} \left(|a_n| \prod_{m_n<k\le n} |\alpha_{k,n}|\right)^{2(n-1)},
\end{align}
where we used that $|\alpha_{j,n}-\alpha_{k,n}| \le 2\max(|\alpha_{j,n}|,|\alpha_{k,n}|)=2|\alpha_{k,n}|.$ Equation \eqref{2.2} implies that
\begin{align} \label{5.2}
\limsup_{n\to\infty} \left(|a_n| \prod_{m_n<k\le n} |\alpha_{k,n}| \right)^{1/n} \le \limsup_{n\to\infty} \left(M(P_n)\right)^{1/n} = H < \infty.
\end{align}
Note that $\liminf_{n\to\infty} m_n/n = \liminf_{n\to\infty} \tau_n(\overline D_R) = \tau(\overline D_R)$ by Theorem 2.1 of \cite{Bi}, as $\tau(\partial D_R)=0$. Thus we obtain from \eqref{5.1}-\eqref{5.2} that
\begin{align} \label{5.3}
\liminf_{n\to\infty} |V(\hat P_n)|^{\frac{2}{(n-1)n}} &\ge 4^{\tau(\overline D_R)-1}\ \liminf_{n\to\infty} \left(|a_n| \prod_{m_n<k\le n} |\alpha_{k,n}| \right)^{-2/n} \\ &\ge 4^{\tau(\overline D_R)-1} H^{-2}. \nonumber
\end{align}
Let $K_M(z,t) := \min\left(-\log{|z-t|},M\right).$ It is clear that $K_M(z,t)$ is a continuous function in $z$ and $t$ on $\C\times\C$, and that $K_M(z,t)$ increases to $-\log|z-t|$ as $M\to\infty.$ Using the Monotone Convergence Theorem and the weak* convergence of $\hat\tau_n\times\hat\tau_n$ to $\tau_R\times\tau_R,$ we obtain for the energy of $\tau_R$ that
\begin{align*}
I[\tau_R] &=-\iint \log|z-t|\,d\tau_R(z)\,d\tau_R(t) \\ &= \lim_{M\to\infty} \left( \lim_{n\to\infty} \iint K_M(z,t)\, d\hat\tau_n(z)\,d\hat\tau_n(t) \right) \\ &= \lim_{M\to\infty} \left( \lim_{n\to\infty} \left( \frac{2}{n^2} \sum_{1\le j<k\le m_n} K_M(\alpha_{j,n},\alpha_{k,n}) + \frac{M}{n} \right) \right) \\ &\le \lim_{M\to\infty} \left( \liminf_{n\to\infty} \frac{2}{n^2}
\sum_{1\le j<k\le m_n} \log\frac{1}{|\alpha_{j,n}-\alpha_{k,n}|} \right) \\ &= \liminf_{n\to\infty} \frac{2}{n^2}
\log\frac{1}{|V(\hat P_n)|} \le (1-\tau(\overline D_R))\log{4} + 2 \log{H},
\end{align*}
where \eqref{5.3} was used in the last estimate. This gives the upper bound in \eqref{2.4}.

We continue with the same notation to prove \eqref{2.3}. Note that
\[
\int\log^+|z|\,d\hat\tau_n = \frac{1}{n} \log  \prod_{k=1}^{m_n} \max(1,|\alpha_{k,n}|) \le \log \left(M(P_n)\right)^{1/n}.
\]
Extending $\log^+|z|$ from $\overline D_R$ to a continuous nonnegative function with compact support in $\C$, we use the weak* convergence $\hat\tau_n \stackrel{*}{\rightarrow} \tau_R$ and \eqref{2.2} to obtain
\[
\int\log^+|z|\,d\tau_R = \lim_{n\to\infty} \int \log^+|z|\,d\hat\tau_n \le  \limsup_{n\to\infty} \log \left(M(P_n)\right)^{1/n} = \log{H}.
\]
Thus \eqref{2.3} follows from the above inequality by letting $R\to\infty.$

The lower bound of \eqref{2.4} is a consequence of \eqref{2.3} and the estimate below, where we also use that $|z-t|\le 2\max(|z|,|t|)$ and $\tau_R(\overline D_R) \le 1.$
\begin{align*}
I[\tau_R] & = -\iint \log|z-t|\,d\tau_R(z)\,d\tau_R(t) \\ &\ge -\iint \log(2\max(|z|,|t|))\,d\tau_R(z)\,d\tau_R(t) \\ &\ge -\log{2} - 2 \int \left(\int_{|z|\ge |t|} \log|z|\,d\tau_R(z)\right)d\tau_R(t) \\ &\ge -\log{2} - 2 \int \left(\int_{|z|\ge \max(1,|t|)} \log|z|\,d\tau_R(z)\right)d\tau_R(t) \\ &\ge -\log{2} - 2 \int \log{H}\,d\tau_R(t) = -\log{2} - 2 \tau_R(\overline D_R) \log{H}.
\end{align*}

\end{proof}

\begin{proof}[Proof of Corollary \ref{cor2.2}]

We arrange the numbers $\alpha_{k,n}$ for this proof as follows:
\[
|\alpha_{1,n}| \ge |\alpha_{2,n}| \ge \ldots \ge |\alpha_{K_n,n}| \ge 1 > |\alpha_{K_n+1,n}| \ge \ldots \ge |\alpha_{n,n}|.
\]
Starting with the case $m=1$, we apply the arithmetic-geometric mean inequality to obtain that
\begin{align} \label{5.4}
\left(\prod_{k=1}^{K_n} |\alpha_{k,n}|\right)^{1/n} &\le \left(\frac{1}{K_n} \sum_{k=1}^{K_n} |\alpha_{k,n}|\right)^{K_n/n}  \\ \nonumber &\le  \left(\frac{n}{K_n}\right)^{K_n/n}  \left(\frac{1}{n} \sum_{k=1}^n |\alpha_{k,n}|\right)^{K_n/n}\\ \nonumber &\le e^{1/e}  \left(\frac{1}{n} \sum_{k=1}^n |\alpha_{k,n}|\right)^{K_n/n},
\end{align}
where we also used the fact $\max_{x\in[0,\infty)} x^{1/x} = e^{1/e}$ with $x=n/K_n.$ Hence
\begin{align*}
\limsup_{{\mathcal N}\ni n\to\infty} \left(M(P_n)\right)^{1/n}  &\le \limsup_{{\mathcal N}\ni n\to\infty} |a_n|^{1/n}\ \limsup_{{\mathcal N}\ni n\to\infty} \left(\prod_{k=1}^{K_n} |\alpha_{k,n}| \right)^{1/n} < \infty
\end{align*}
by the assumptions and \eqref{5.4}, so that \eqref{2.2} holds. The conclusion of this corollary for $m=1$ now follows from Theorem \ref{thm2.1}.

If the symmetric means in the assumptions are bounded for a fixed $m\ge 2$, then we consider two different cases. Suppose first that $m\le K_n\le n.$ We argue similarly to \eqref{5.4}, and apply \eqref{2.1} to estimate
\begin{align*}
\left(\prod_{k=1}^{K_n} |\alpha_{k,n}|\right)^{1/n} &\le \left(S_m(|\alpha_{1,n}|,\ldots,|\alpha_{K_n,n}|)\right)^{K_n/(mn)} \\ &\le  \left(\binom{n}{m} / \binom{K_n}{m}\right)^{K_n/(mn)}  \left(S_m(|\alpha_{1,n}|,\ldots,|\alpha_{n,n}|)\right)^{K_n/(mn)}.
\end{align*}
Note that
\begin{align*}
\left(\binom{n}{m} / \binom{K_n}{m}\right)^{K_n/(mn)} &\le \left(\frac{n}{K_n-m+1}\right)^{K_n/n} \le \left(\frac{n}{K_n}\right)^{K_n/n} \frac{K_n}{K_n-m+1} \\ &\le e^{1/e} m.
\end{align*}
Hence we obtain that $\limsup_{n\to\infty} \left(M(P_n)\right)^{1/n} < \infty$ as before, and apply Theorem \ref{thm2.1}.

The second remaining case is when $K_n< m\le n.$ We can assume that the constant terms $a_0$ in polynomials $P_n$ are non-zero, for otherwise we can replace $P_n$ with $P_n(z)/z$. Thus we assume that $|a_0|\ge 1,$ which implies that
\begin{align} \label{5.5}
\prod_{k=K_n+1}^n |\alpha_{k,n}| = \frac{|a_0|}{|a_n|\prod_{k=1}^{K_n} |\alpha_{k,n}|} \ge \frac{1}{|a_n|\prod_{k=1}^{K_n} |\alpha_{k,n}|}.
\end{align}
For $n\ge 2m$, we have that
\[
\prod_{k=K_n+1}^n |\alpha_{k,n}| = \prod_{k=K_n+1}^m |\alpha_{k,n}| \prod_{k=m+1}^n |\alpha_{k,n}| \le \left(\prod_{k=K_n+1}^m |\alpha_{k,n}|\right)^2.
\]
It follows from \eqref{5.5} and the above estimate that
\[
\prod_{k=1}^{K_n} |\alpha_{k,n}| = \frac{\prod_{k=1}^m |\alpha_{k,n}|}{\prod_{k=K_n+1}^m |\alpha_{k,n}|} \le \left(|a_n|\prod_{k=1}^{K_n} |\alpha_{k,n}|\right)^{1/2} \prod_{k=1}^m |\alpha_{k,n}|
\]
and therefore
\[
\prod_{k=1}^{K_n} |\alpha_{k,n}| \le |a_n| \left(\prod_{k=1}^m |\alpha_{k,n}|\right)^2 \le |a_n| \left(\binom{n}{m} \ S_m(|\alpha_{1,n}|,\ldots,|\alpha_{n,n}|)\right)^2.
\]
Finally, we obtain that
\begin{align*}
\limsup_{{\mathcal N}\ni n\to\infty} \left(\prod_{k=1}^{K_n} |\alpha_{k,n}| \right)^{1/n} &\le \limsup_{{\mathcal N}\ni n\to\infty} |a_n|^{1/n}\ \limsup_{{\mathcal N}\ni n\to\infty} \ \left(S_m(|\alpha_{1,n}|,\ldots,|\alpha_{n,n}|)\right)^{2/n} \\ &< \infty,
\end{align*}
so that $\dis\limsup_{{\mathcal N}\ni n\to\infty} \left(M(P_n)\right)^{1/n} < \infty$ and the result follows from Theorem \ref{thm2.1} again.

\end{proof}

\begin{proof}[Proof of Theorem \ref{thm2.3}]

We show how to approximate an arbitrary Borel measure $\mu,\ 0\le\mu(\C)\le 1,$ by the counting measures of algebraic integers in the weak* topology. This approximation is sketched below in several steps of reduction for the problem. The case $\mu(\C)=0$ is covered by the example given after Theorem \ref{thm2.3}. Thus we assume that $\mu$ is supported in a compact set $E$ and $\mu(E)>0$ without loss of generality. Otherwise one can apply constructed below approximations to the restrictions of $\mu$ on a family of expanding disks filling the plane, exactly as mentioned before Theorem \ref{thm2.1}, and obtain the desired weak* convergent approximations for $\mu$ in $\C.$

The first step of our approximation scheme uses the Krein-Milman theorem (see \cite[pp. 362--363]{Yo}) to express $\mu$ as the weak* limit of linear convex combinations of point masses of the form
\[
\sum_{j=1}^J t_j \delta_{z_j}, \quad
t_j > 0, \quad \sum_{j=1}^J t_i = \mu(E)=\mu(\C)\le 1,
\]
where $\{z_j\}^{J}_{j=1} \subset E$ are selected symmetric about the real line. The latter requirement of symmetry is achieved by applying the Krein-Milman theorem to the restriction of $\mu$ to the upper half-plane, and adding the reflection of thus obtained discrete measure to approximate $\mu$ in the lower half-plane part of its support.

The set of coefficients $t_j\in\R,\ j=1,\ldots,J,$ can be simultaneously approximated by the rational numbers $k_j/L < t_j$ with the same denominator $L\in\N$. Therefore, we can select a set of $k_j$ points $w_{i,j},\ i=1,\ldots,k_j,$ on a sufficiently small circle around $z_j$, for each $j=1,\ldots,J,$ so that the measure $\frac{1}{L}\sum_{i=1}^{k_j} \delta_{w_{i,j}}$ gives an arbitrary close approximation to $t_j \delta_{z_j}.$ Furthermore, this can clearly be done so that all points $w_{i,j}$ are distinct and symmetric about the real axis. It remains to approximate the resulting discrete measure
\[
\nu_K = \frac{1}{L} \sum_{j=1}^J \sum_{i=1}^{k_j} \delta_{w_{i,j}},\quad K:=\sum_{j=1}^J k_j,
\]
by the counting measures for algebraic integers on the third step.

Note that the measure $\nu_K$ is a counting (but not unit) measure for the points $w_{i,j}$ totalling $K < L$. Therefore, we can approximate this measure by a sequence of the counting measures $\tau_{K+1}$ for the complete sets of conjugate algebraic integers $\{\alpha_{l,K+1}\}_{l=1}^{K+1}$ by using the theorem of Motzkin \cite{Mo}, see also \cite{Du} for its effective version. For any $K\in\N$, we approximate each point $w_{i,j}$ as close as we wish by one of the conjugate algebraic integers $\alpha_{l,K+1},\ 1\le l \le K,$ obtained from Motzkin's theorem, while let the remaining $(K+1)$st conjugate algebraic integer $\alpha_{K+1,K+1}\to\infty$ as $K\to\infty$ (cf. \cite[p. 160--161]{Mo} for details). It follows that the resulting measures
\[
\tau_{K+1} = \frac{1}{L} \sum_{l=1}^{K+1} \delta_{\alpha_{l,K+1}}
\]
converge to $\mu$ in the weak* topology as $K\to\infty.$

\end{proof}

\begin{proof}[Proof of Theorem \ref{thm2.4}]

Assume that $\{z_{k,n}\}_{k=1}^n\subset D_R$ and that $\tau_n \stackrel{*}{\rightarrow} \tau$ for a subsequence $n\in{\mathcal N}\subset\N.$ Multinomial theorem gives
\begin{align*}
\left(\sum_{j=1}^n z_{j,n}\right)^m &= \sum_{k_1+k_2+\ldots+k_n=m} \binom{m}{k_1,k_2,\ldots,k_n} z_{1,n}^{k_1} z_{2,n}^{k_2} \ldots z_{n,n}^{k_n} \\ &= m!\,\sigma_m(z_{1,n}, z_{2,n}, \ldots, z_{n,n}) \\ &+ \sum_{k_1+k_2+\ldots+k_n=m \atop {\exists\ k_j \ge 2}} \binom{m}{k_1,k_2,\ldots,k_n} z_{1,n}^{k_1} z_{2,n}^{k_2} \ldots z_{n,n}^{k_n}.
\end{align*}
Note that the sum $s$ of multinomial coefficients in the latter summation can be found by setting $z_{j,n}=1,\ j=1,\ldots,n,$ in the above formula, which gives
\[
n^m= m! \binom{n}{m} + s = \prod_{j=0}^{m-1} (n-j) +s.
\]
It follows that $s=O\left(n^{m-1}\right)$ as $n\to\infty$, and that
\[
\left|\left(\sum_{j=1}^n z_{j,n}\right)^m - m!\,\sigma_m(z_{1,n}, \ldots, z_{n,n})\right| \le R^{m}\, O\left(n^{m-1}\right) \quad\mbox{as } n\to\infty.
\]
Since $\lim_{n\to\infty} m!\,\binom{n}{m}/n^m = 1$, we can divide the above equation by $n^m$ and let $n\to\infty$ to obtain the first equality in \eqref{2.7}, provided one of the limits in \eqref{2.7} exists. But the fact
\[
\lim_{{\mathcal N}\ni n\to\infty} S_1(z_{1,n},\ldots,z_{n,n}) = \int z\,d\tau(z)
\]
is an easy consequence of the weak* convergence $\tau_n \stackrel{*}{\rightarrow} \tau$ for $n\in{\mathcal N}.$

\end{proof}

\begin{proof}[Proof of Theorem \ref{thm2.5}]

Consider the half-plane $H_a=\{z\in\C: \Re\,z \le a\}.$ Given a fixed $a>0$, we arrange each set $\{z_{k,n}\}_{k=1}^n$ in the order of increasing real parts:
\[
\Re\,z_{1,n} \le \Re\,z_{2,n} \le \ldots \le \Re\,z_{l_n,n} \le a < \Re\,z_{l_n+1,n} \le \ldots \le \Re\,z_{n,n}.
\]
Define $\hat\tau_n:=\tau_n\vert_{H_a}.$ It follows that
\[
\hat\tau_n \stackrel{*}{\rightarrow} \tau_a:=\tau\vert_{H_a},\ n\in{\mathcal N}, \quad \mbox{and} \quad \lim_{{\mathcal N}\ni n\to\infty} l_n/n = \tau(H_a)
\]
 for any $a>0$ such that $\tau(\{z:\Re\,z=a\})=0$, by Theorem $0.5^\prime$ of \cite{La} and Theorem 2.1 of \cite{Bi}. Hence the above equation holds for all but countably many $a>0$, and we consider only such $a$ below.

Observe that $S_m(z_{1,n}, z_{2,n}, \ldots, z_{n,n})\ge 0$ and $S_m(z_{1,n}, z_{2,n}, \ldots, z_{l_n,n})\ge 0$ for $\{z_{k,n}\}_{k=1}^n\subset V_m$ because all terms in these sums have positive real parts, and their imaginary parts cancel due to symmetry of $\{z_{k,n}\}_{k=1}^n$. For the same reason, we have that $\sigma_m(z_{1,n}, \ldots, z_{n,n}) \ge \sigma_m(z_{1,n}, \ldots, z_{l_n,n}),$ so that
\[
S_m(z_{1,n}, \ldots, z_{n,n}) \ge  S_m(z_{1,n}, \ldots, z_{l_n,n}) \binom{l_n}{m}/\binom{n}{m}.
\]
Since
\[
\lim_{n\to\infty} \binom{l_n}{m}/\binom{n}{m} = \lim_{n\to\infty} \left(l_n/n\right)^m = \left(\tau(H_a)\right)^m,
\]
we obtain by the weak* convergence $\hat\tau_n \stackrel{*}{\rightarrow} \tau_a,\ n\in{\mathcal N},$ and Theorem \ref{thm2.4} that
\begin{align*}
\liminf_{{\mathcal N}\ni n\to\infty} S_m(z_{1,n}, \ldots, z_{n,n}) &\ge \left(\tau(H_a)\right)^m  \liminf_{{\mathcal N}\ni n\to\infty} S_m(z_{1,n}, \ldots, z_{l_n,n}) \\ &= \left(\tau(H_a)\right)^m \left(\int z\,d\tau_a(z)\right)^m.
\end{align*}
The concluding step is to let $a\to+\infty$ in the above inequality, using
\[
\lim_{a\to+\infty} \tau(H_a) = \tau(\C) = 1 \quad \mbox{and} \quad \lim_{a\to+\infty} \int z\,d\tau_a(z) = \int z\,d\tau(z).
\]

\end{proof}

\begin{proof}[Proof of Theorem \ref{thm2.6}]

Equation \eqref{2.9} implies that $\tau_n \stackrel{*}{\rightarrow} \mu_E$ by Theorem 2.1 of \cite{PrCrelle}, where $\mu_E$ is the equilibrium measure of $E$. Hence we obtain that
\[
\liminf_{{\mathcal N}\ni n\to\infty} S_m(\alpha_{1,n}, \ldots, \alpha_{n,n}) \ge \left(\int z\,d\mu_E(z)\right)^m,\quad m\in\N,
\]
by Theorem \ref{thm2.5}. Lower bounds for the moments of equilibrium measures were established in \cite{BLP}. We apply the change of variable $x=t^2$, and define the compact set $K=\{t\in\R: t^2\in E\}$. Then $K$ is symmetric about the origin, so that $\int t\, d\mu_K(t)=0$. Furthermore, $d\mu_K(t) = d\mu_E(t^2),\ t\in K$, and cap$(K)=1$; see \cite[p. 134]{Ra}. It now follows from Theorem 1 of \cite{BLP} that
\[
\int x\, d\mu_E(x) = \int t^2\, d\mu_K(t) \ge \int t^2\, d\mu_{[-2,2]}(t) = \int_{-2}^2
\frac{t^2\,dt}{\pi\sqrt{4-t^2}} = 2,
\]
so that \eqref{2.10} is proved.

We discussed in Section 1 that the Chebyshev polynomials $t_n(x):=2\cos(n\arccos((x-2)/2))$ for $[0,4]$ satisfy the assumptions of this theorem. Moreover, the counting measures $\tau_n \stackrel{*}{\rightarrow} \mu_{[0,4]} = dx/(\pi\sqrt{x(4-x)})$ as $n\to\infty$, which gives
\[
\lim_{n\to\infty} S_1(\alpha_{1,n}, \ldots, \alpha_{n,n}) = \int_0^4 \frac{x\,dx}{\pi\sqrt{x(4-x)}} = 2.
\]
Combining this with \eqref{2.1} and \eqref{2.10}, we immediately see that equality holds in \eqref{2.10} for all $m\in\N.$

\end{proof}

\begin{proof}[Proof of Theorem \ref{thm3.2}]

Using convexity of $\phi$, we have that
\begin{align*}
\limsup_{{\mathcal N}\ni n\to\infty} \phi\left(\frac{1}{n} \sum_{k=1}^n |\alpha_{k,n}|\right) \le \limsup_{{\mathcal N}\ni n\to\infty} C_n^{\phi}(|\alpha_{1,n}|,\ldots,|\alpha_{n,n}|) < \infty.
\end{align*}
Since $\phi$ is also strictly increasing, we conclude that $\lim_{x\to+\infty} \phi(x) = +\infty.$ Hence we obtain by the above inequality that
\[
\limsup_{{\mathcal N}\ni n\to\infty} \frac{1}{n} \sum_{k=1}^n |\alpha_{k,n}| < \infty.
\]
Thus all assumptions of Corollary \ref{cor2.2} are satisfied with $m=1$, and the desired conclusion follows.

\end{proof}

\begin{proof}[Proof of Theorem \ref{thm4.1}]

Since \eqref{2.6} is satisfied, Corollary \ref{cor2.2} gives that any weak* limit of the counting measures $\tau_n$ for our sequence must be a unit measure with finite logarithmic energy.
The inequality $\gamma\le \pi/(2m)$ insures that $W_{\gamma}\subset V_m$ for each such $m\in\N.$ Hence all conditions of Theorem \ref{thm2.5} are satisfied, and \eqref{2.8} gives us \eqref{4.1}, where $\tau$ is a weak* limit of $\tau_n$ along a subsequence that attains the value of $\liminf_{{\mathcal N}\ni n\to\infty} S_m(\alpha_{1,n}, \ldots, \alpha_{n,n}).$ Note that $\tau$ is supported in $W_{\gamma}$, and is different from the point mass $\delta_0$ at the origin, because $\tau$ has finite energy. Hence $\int z\,d\tau(z) > 0.$

\end{proof}

\subsection*{Acknowledgements}
This research was partially supported by the National Security Agency (grant H98230-12-1-0227), and by the AT\&T Professorship.

\end{document}